\documentclass[a4paper,oneside,reqno]{amsart}    
\usepackage[foot]{amsaddr}
\usepackage[utf8]{inputenc}
\usepackage[english]{babel}
\usepackage[T1]{fontenc}
\usepackage{amsfonts}
\usepackage{amsmath}
\usepackage{graphicx}
\usepackage{amssymb}
\usepackage{amsopn}
\usepackage{amsthm}
\usepackage{mathtools}
\usepackage{color}
\usepackage{hyperref}

\allowdisplaybreaks

\newtheorem{theorem}{Theorem}

\newtheorem{definition}{Definition}
\newtheorem{proposition}{Proposition}
\newtheorem{remark}{Remark}
\theoremstyle{definition}
\newtheorem{example}{Example}

\newcommand{\wt}[1]{\widetilde{#1}}

\newcommand{\K}{\mathcal{K}}
\newcommand{\cL}{\mathcal{L}}
\newcommand{\KL}{\mathcal{KL}}

\newcommand{\R}{\mathbb{R}}
\newcommand{\N}{\mathbb{N}}
\newcommand{\eps}{\varepsilon}
\newcommand{\sat}{\mathfrak{sat}}
\newcommand{\dxi}{\frac{\mathrm{d}}{\mathrm{d}\xi}}
\newcommand{\e}{\mathrm{e}}

\newcommand{ \satr}{\mathrm{sat}_\R}

\begin{document}

\title[Remarks on ISS for saturated systems]{Remarks on input-to-state stability 
of collocated systems with saturated feedback}

\author{Birgit Jacob}        

\address[BJ, LV]{University of Wuppertal, School of Mathematics and Natural Sciences, Gau\ss stra\ss e 20, D-42119 Wuppertal, Germany.}
 \email{bjacob@uni-wuppertal.de, vorberg@uni-wuppertal.de}   
\author{Felix L.~Schwenninger}
\address[FLS]{Department of Applied Mathematics, University of Twente, P.O.~Box 217, 
7500 AE Enschede, 
The Netherlands and \newline Department of Mathematics, Center for Optimization and Approximation, University of Hamburg, 
Bundesstr. 55, 20146 Hamburg, Germany}
\email{f.l.schwenninger@utwente.nl}
\author{Lukas A.~Vorberg}



\date{August 2020}

\begin{abstract}
{We investigate input-\-to-\-state stability (ISS) of infinite-\-dimensional collocated  control systems subject to saturated feedback. Here, the unsaturated closed loop is dissipative and uniformly globally asymptotically stable. Under an additional assumption on the linear system, we show ISS for the saturated one. We discuss the sharpness of the conditions in light of existing results in the literature. }
\end{abstract}
\maketitle

\section{Introduction}
\label{intro}
In this note we continue the study of the stability of  systems of the form 
\begin{align*}
\begin{cases}
\dot{x}(t) = Ax(t)-B\sigma\big(B^*x(t)+d(t)\big), \\
x(0) = x_0,
\end{cases}
\tag{$\Sigma_{SLD}$}
\label{System SLD}
\end{align*} 
derived from the linear collocated open-loop system
\begin{align*}
\dot{x}(t) &= Ax(t) +B u(t), \\
y(t) &= B^* x(t).
\end{align*}
by the nonlinear feedback law $u(t)=-\sigma(y(t)+d(t))$. 
Here $X$ and $U$ are Hilbert spaces, $A:D(A)\subset X\rightarrow X$ is the generator of a  strongly continuous contraction semigroup and $B$ is a bounded linear operator from $U$ to $X$, i.e. $B\in\mathcal{L}(U,X)$.
The function $\sigma:U\rightarrow U$ is locally Lipschitz continuous and maximal monotone 
with $\sigma(0)=0$. Of particular interest is the case in which $\sigma$ is even linear in a neighbourhood of $0$. The open-loop system is called collocated as the output operator $B^*$ equals the adjoint  of the input operator $B$.
In the following we are interested in stability with respect to both the initial value $x_{0}$, that is {\it internal stability}, and the disturbance $d$; {\it external stability}. This is combined in the notion of {\em input-to-state stability} (ISS), which has recently been studied for infinite-dimensional systems e.g.\ in \cite{GuLoOp19,JNPS18,mironchenkowirth18a,mironchenkowirth18b} and particularly for semilinear systems in  \cite{Gr98,GuLo18,Schw20}, see also \cite{MiroPrie19} for a survey. The effect of feedback laws acting (approximately) linearly only locally is known in the literature as {\it saturation}, and first appeared in \cite{MR997216,MR1808404} in the context of stabilization of infinite-dimensional linear systems, see also \cite{article4}. There, internal stability of the closed-loop system was studied using nonlinear semigroup theory, a natural tool to establish existence and uniqueness of solutions for equations of the above type, see also the more recent works \cite{MaAnPr17,MaChPr18,MaChPr19}. The simultaneous study of internal stability and the robustness with respect to additive disturbances in the saturation seems to be rather recent. This notion clearly includes uniform global (internal) stability, which is far from being trivial for such nonlinear systems. In \cite{7386612} this was studied for a wave equation and in \cite{marx:hal-01367622} Korteweg-de Vries type equation was rigorously discussed, building on preliminary works in \cite{marx:hal-01193019,marx:hal-01360576}, see also \cite{MaAnPr17}. \\
 The combination of saturation and ISS was initiated in \cite{MaChPr18} and,  as for internal stability, complemented in \cite{MaChPr19}.  For the rich finite-dimensional theory on ISS for related semilinear systems, we refer e.g.\ to \cite{Gr98,GuLo18} and the references therein. For (infinite-dimensional) nonlinear systems, ISS is typically assessed by Lyapunov functions, see e.g.\ \cite{DaM13,JMPW20,Mi16,mironchenkowirth18b,Schw20}. These are often constructed by energy-based $L^{2}$ norms, but also Banach space methods exist \cite{mironchenkowirth18b}, which are much easier to handle in the sense of $L^{\infty}$-estimates as present in ISS. We will use some of these constructions here. 

In this note we investigate the question whether internal stability of the linear undisturbed system, that is, \eqref{System SLD} with $\sigma(u)=u$ and $d\equiv 0$, implies input-to-state stability of \eqref{System SLD}. In doing so we try to shed light on limitations of existing results. Because the linear system has a bounded input operator, the above question is equivalent to asking whether ISS of the linear system yields that  \eqref{System SLD} is ISS, see e.g.\ \cite{JNPS18}.
 For nonlinear systems, uniform global asymptotic (internal) stability is only a necessary condition for ISS, which, however, may fail in presence of saturation. 
Indeed, the following saturated transport equation will serve as a model for a counterexample which we shall discuss in this note in detail, see Theorem \ref{theo:sat}.
\begin{align*}\label{System sat}
\begin{cases}
\dot{x}(t,\xi) = \dxi x(t,\xi)-\mathrm{sat}_\R\big(x(t,\xi)\big),\quad (t,\xi)\in(0,\infty)\times [0,1], \\
x(t,0)= x(t,1),\\
x(0,\xi) = f(\xi),
\end{cases}
\tag{$\Sigma_\sat$}
\end{align*}
where 
\begin{align}\label{eq:sat}
{\mathrm{sat}_\R}(z) \coloneqq \begin{cases} \frac{z}{|z|}, \quad &|z|\geq1 \\
z, \quad &z\in(-1,1).
\end{cases}
\end{align}

\section{ISS for saturated systems}
\label{sec:1}

\begin{definition}\label{def:admfeedback}
We call $\sigma:U\rightarrow U$ an {\em admissible feedback function} if 
\begin{enumerate}
\item $\sigma(0) = 0$,
\item $\sigma$ is \emph{locally Lipschitz continuous}, i.e. for every $r>0$ there exists a $k_r>0$ such that
\begin{align*}
\|\sigma(u)-\sigma(v)\|_U \leq k_r\|u-v\|_U  \quad \forall \ u,v \in U \text{ with } \|u\|_U,\|v\|_U\leq r,
\end{align*}
\item $\sigma$ is {\em maximal monotone}, i.e.\ $\Re\langle \sigma(u)-\sigma(v),u-v\rangle_U \geq 0 \quad \forall \ u,v\in U$.
\end{enumerate}
If additionally a Banach space $S$ is  continuously, densely embedded in $U$ with dual space $S'$ such that
\begin{enumerate}
\setcounter{enumi}{3}
\item $\|\sigma(u)-u\|_{S'}\leq \Re\langle\sigma(u),u\rangle_U \quad \forall \ u\in U$, and
\item there exists $C_0>0$ such that
\begin{align*}
\Re\langle u,\sigma(u+v)-\sigma(u)\rangle_U \leq C_0\|v\|_U \quad \forall \ u,v\in U,
\end{align*} 
\end{enumerate}
then we call $\sigma$ a {\em saturation function}. Here $U\subset S'$ is understood in the sense of rigged Hilbert spaces, i.e.\ an element $u$ in $U$ is identified with the functional $s\mapsto \langle s,u\rangle_U$  in $S'$. 
\end{definition}

It seems that the notion of a saturation function appeared first in the context of infinite-dimensional systems in \cite{MR1808404,MR997216}. Note that the precise definition --- in particular which properties it should include --- has varied in the literature since then. 
Our definition here matches the one in \cite{MaChPr18}, except for the fact that there, in addition, it is required that $\|\sigma(u)\|_{S}\leq1$.
We distinguish between ``admissible feedback functions'' and ``saturation functions'' in order to point out which (minimal) assumptions are needed in the following results.

\begin{example}\label{Exmp sat}
Let $\mathrm{sat}_\R$ be the function from \eqref{eq:sat}. It is easy to see that the function 
\begin{align*}
\sat :L^2(0,1)\rightarrow L^2(0,1), \quad u\mapsto \mathrm{sat}_\R(u(\cdot))
\end{align*}
is an admissible feedback function. Moreover, for $S=L^\infty(0,1)$  we have that
\begin{align*}
\|\sat(u) - u\|_{L^1(0,1)} &= \int_0^1|\sat(u)(\xi)-u(\xi)| \, \mathrm{d}\xi \\
&\leq \int_{\{u\geq1\}}u(\xi) \, \mathrm{d}\xi + \int_{\{-1\leq u\leq 1\}}u^2(\xi) \, \mathrm{d}\xi+\int_{\{u\leq-1\}}-u(\xi) \, \mathrm{d}\xi \\
&=\langle\sat(u),u\rangle_U \qquad \forall u\in U.
\end{align*}
As Property (v) from Definition \ref{def:admfeedback} follows similarly, 
 $\sat$ is a saturation function. Note that this example is well-known in the literature, see \cite{MaChPr18,MaChPr19} and the references therein.
\end{example}

Let $\sigma$ be an admissible feedback function. In the rest of the paper we will be interested in the following two types of systems:
The {\em unsaturated system},
\begin{align*}
\begin{cases}
\dot{x}(t) = Ax(t) -BB^*x(t), \\
x(0) = x_0,
\end{cases}
\tag{$\Sigma_L$}
\label{System L}
\end{align*}
and the {\em disturbed saturated system}
\begin{align}\label{SLD}
\begin{cases}
\dot{x}(t) = Ax(t)-B\sigma\big(B^*x(t)+d(t)\big), \\
x(0) = x_0.
\end{cases}
\tag{$\Sigma_{SLD}$}
\end{align} 
 with $d\in L^\infty(0,\infty;U)$. We abbreviate $$\wt{A}:D(\wt{A})\subset X\rightarrow X, \quad\wt{A}x:= Ax-BB^* x.$$
By the Lumer--Phillips theorem,  $\wt{A}$ generates a strongly continuous semigroup of contractions $(\wt{T}(t))_{t\geq 0}$ 
as $-BB^*\in L(X)$ is dissipative.
Moreover, the nonlinear operator $A-B\sigma(B^{*}\cdot)$ generates a nonlinear semigroup of contractions \cite[Thm.~1]{Webb} since, obviously, $B\sigma(B^{*}\cdot):X\to X$ is continuous and monotone, i.e.\ 
\[\langle B\sigma(B^{*}x)-B\sigma(B^{*}y),x-y\rangle \ge 0, \qquad \forall x,y\in X.\]
Clearly, \eqref{System L} is a special case of \eqref{System SLD} with $d=0$, as $\sigma(u)=u$ is an admissible feedback function. 

\begin{definition}
Let $x_{0}\in X$, $d\in L_{\mathrm{loc}}^\infty(0,\infty;U)$ and $t_{1}>0$. A continuous function $x:[0,t_{1}]\to X$ satisfying 
\begin{align*}
x(t) = T(t)x_0-\int_0^tT(t-s)B\sigma\big(B^*x(s)+d(s)\big)\,\mathrm{d}s,\quad t\in [0,t_{1}],
\end{align*}
is called a \emph{mild solution} of \eqref{SLD} on $[0,t_{1}]$ and we may omit the reference to the interval. If $x:[0,\infty)\to X$ is such that the restriction $x|_{[0,t_{1}]}$ is a mild solution for every $t_{1}>0$, then $x$ is called a \emph{global mild solution}.
\end{definition}

By our assumptions,  \eqref{System SLD} has a unique mild solution (on some maximal interval) for any $x_{0}\in X$  and $d\in L^\infty(0,\infty;U)$, \cite[Thm.~6.1.4]{pazy83}\footnote{A careful look at the proof reveals that the continuity of the nonlinearity in $t$ required in \cite[Thm.~6.1.2]{pazy83} can be dropped in our setting.}.
In order to introduce the external stability notions,  the following well-known comparison functions are needed,
\begin{align*}
\K &\coloneqq \{\alpha\in C(\R_+,\R_+) \ | \ \alpha \text{ is strictly increasing}, \alpha(0)=0 \},\\
\K_\infty &\coloneqq \{ \alpha\in\K \ | \ \alpha \text{ is unbounded}\}, \\
\cL &\coloneqq \{\alpha\in C(\R_+,\R_+) \ | \ \alpha \text{ is strictly decreasing with } \lim_{t\rightarrow\infty}\alpha(t)=0\}, \\
\KL &\coloneqq \{ \beta\in C(\R_+\times \R_+,\R_+) \ | \ \beta(\cdot,t)\in\K \ \ \forall t>0, \ \beta(r,\cdot)\in\cL \ \ \forall r>0\},
\end{align*}
where $C(\R_+,\R_{+})$ refers to the continuous functions from $\R_{+}$ to $\R_{+}$.

\begin{definition}
\begin{enumerate}
\item \eqref{System SLD} is called  
{\em{globally asymptotically stable} (GAS)}
if every mild solution $x$ for $d=0$ is global and the following two properties hold;
 $\lim_{t\rightarrow \infty}\|x(t)\|_X=0$ for every initial condition $x_{0}\in X$ and 
there exist $\sigma\in \mathcal{K}_{\infty}$ and $r>0$  such that  $\|x(t)\|\leq\sigma(\|x_{0}\|)$ for every  $x_{0}\in X$ with $\|x_0\|\leq r$, $d=0$ and  $t\ge 0$.
\item \eqref{System SLD} is called \emph{semi-globally exponentially stable in $D(A)$} if for $d=0$ and any $r>0$ there exist $\mu(r)>0$ and $K(r)>0$ such that any mild solution $x$ with initial value $x_0\in D(A)$ is global and satisfies
\begin{align*}
\|x(t)\|_X \leq K(r)\e^{-\mu(r)t}\|x_0\|_X \qquad \forall t\geq0
\end{align*}
 for $\|x_0\|_{D(A)} \coloneqq \|x_0\|_X + \|Ax_0\|_X \leq r$.
\item \eqref{System SLD} is called {\em locally input-to-state stable (LISS)}  if there exist $r>0$, $\beta\in\KL$ and $\rho\in\K_\infty$ such that every mild solution $x$ with initial value satisfying $\|x_0\|_X\leq r$ and disturbance $d$ with $\|d\|_{L^{\infty}(0,\infty;U)}\leq r$ is global and for 
all $t\geq0$ we have that
\begin{align}\label{ISS}
\|x(t)\|_X \leq \beta(\|x_0\|_X,t)+\rho(\|d\|_{L^\infty(0,t;U)}).
\end{align}
\eqref{System SLD} is called {\em input-to-state stable (ISS)} if $r=\infty$.\\
System \eqref{System SLD} is called LISS with respect to $C(0,\infty;U)$ if the above holds for continuous disturbances only.
If \eqref{ISS} holds for \eqref{System SLD} with $d\equiv0$ and $r=\infty$, the system is called {\em uniformly globally asymptotically stable (UGAS)}, where the uniformity is with respect to the initial values.
\end{enumerate}
\end{definition}
Note that in our notation "UGAS" refers to "0-UGAS" and "GAS" refers to "0-GAS" more commonly used in the literature.
The System   \eqref{System SLD} is globally asymptotically stable if and only if  for every mild solution $x$ for $d=0$ we have $\lim_{t\rightarrow \infty}\|x(t)\|_X=0$. 
This directly follows from the fact that the mild solutions of \eqref{System SLD} with $d=0$ can be represented by a (nonlinear) contraction semigroup, which implies that $\|x(t)\|\leq \|x_{0}\|$ for all $t\ge0$, $x_{0}\in X$.
Compared to the other notions, semi-global exponential stability in $D(A)$ seems to be less common in the literature, but appeared already in the context of saturated systems in \cite{MaChPr19}. The notion of semi-global exponential stability in $X$ was studied in \cite{marx:hal-01367622}.
Note that for the linear System \eqref{System L} UGAS is equivalent to the existence of constants $M,\omega>0$ such that $\|\wt{T}(t)\|_X\leq M\e^{-\omega t}$ for all $t\geq0$, see \cite[Proposition V.1.2]{engelnagel99}. Clearly,  if  \eqref{System SLD} is UGAS, then it is globally asymptotically stable.
We note that semi-global exponential stability in $D(A)$ implies global asymptotical stability since $D(A)$ is dense in $X$ and by the above mentioned fact that the mild solutions are described by a nonlinear contraction semigroup.
Moreover, using again the denseness of $D(A)$ in $X$, the System \eqref{System L} is UGAS if and only if it is semi-globally exponentially stable in $D(A)$.  

Next we investigate the question whether  (semi-)global exponential stability in $D(A)$ or UGAS of System  \eqref{System L} implies (semi-)global exponential stability in $D(A)$ or UGAS of System \eqref{System SLD}.

In  \cite[Theorem 2]{MaAnPr17} it is shown that  global asymptotic stability of \eqref{System L} implies global asymptotic stability of \eqref{System SLD} if
\begin{itemize}
\item $D(A)$ equipped with the norm $\|\cdot\|_{D(A)}= \|\cdot\|_X + \|A\cdot\|_X$ is a Banach space  compactly embedded in $X$ and
\item $\sigma$ is an admissible feedback function with the additional properties that for all $u\in U$, $\Re\langle u,\sigma(u)\rangle =0$ implies $u=0$.
\end{itemize}

Note that the other assumptions of \cite[Theorem 2]{MaAnPr17} are satisfied in our situation if $\sigma$ is globally Lipschitz; this follows again by the fact that the mild solutions are represented by a nonlinear semigroup.
In \cite[Section V]{mironchenkowirth18a} it is shown, that under these conditions and in finite dimensions, i.e. $X=\R^n$ and $U=\R^m$, \eqref{System SLD} is UGAS.

Here we are interested in results for general admissible feedback functions and saturation functions. The following result was proved in \cite{MaChPr19} and \cite{MaChPr18}. 

\begin{proposition}[{{\cite[Theorem 1]{MaChPr18}}, \cite[Theorem 2]{MaChPr19}}]\label{Thm SGES}
Let \eqref{System L} be UGAS and $\sigma:U\to U$ be a globally Lipschitz saturation function. 
\begin{enumerate}
\item If $S=U$, then \eqref{System SLD} is ISS. 
\item\label{prop1ii} If there exists a bounded self-adjoint operator $P$ which maps $D(A)$ to $D(A)$ and solves 
\begin{equation}\label{prop1eq1}
\langle \tilde{A}x,Px\rangle+\langle Px,\tilde{A}x\rangle\leq -\langle x,x\rangle,\qquad \forall x\in D(\tilde{A})=D(A),
\end{equation}
and if 
\begin{align}\label{30}
\exists c>0 \, \forall x\in D(A):\quad \|B^*x\|_S \leq c\|x\|_{D(A)},
\end{align}
then \eqref{System SLD} is semi-globally exponentially stable in $D(A)$.
\end{enumerate}
\end{proposition}
Note that in the second part of Proposition \ref{Thm SGES}, the existence of a bounded, self-adjoint operator $P$ satisfying \eqref{prop1eq1} always follows from the assumption that  \eqref{System L} is UGAS. However, the property that such $P$ leaves $D(\tilde{A})$ invariant does not hold in general. For instance, this is satisfied if there exists $\alpha>0$ such that $\Re\langle Ax,x\rangle \leq -\alpha \|x\|^{2}$  all $x\in D(A)$, which follows directly from dissipativity. On the other hand, it is not hard to construct examples where this invariance is not satisfied. We will comment on this condition also in Remark \ref{rem:thm} ii).
We will show next that Proposition \ref{Thm SGES} ii) does not hold without assuming \eqref{30} and moreover, that \eqref{30} does neither imply UGAS nor ISS for \eqref{System SLD}.

\begin{proposition}\label{Thm counterexmp}
Let $X=U=L^2(0,1)$, $S=L^\infty(0,1)$, $A=0$, $B=I$ and $\sigma=\sat$. Then System  \eqref{System L} is UGAS and System \eqref{System SLD} is neither semi-globally exponentially stable in $D(A)$, nor UGAS nor ISS.
\end{proposition}
\begin{proof}
As System  \eqref{System L} is given by $\dot{x}(t)=-x(t)$, it is UGAS. System  \eqref{System SLD} 
s given by 
\begin{align}\label{System sat0}
\begin{cases}
\dot{x}(t,\xi) = -\satr\big(x(t,\xi)\big), \quad t\ge 0, \xi\in (0,1),\\
x(0,\xi) = f(\xi),
\end{cases}
\end{align}
with the unique mild solution $x\in C([0,\infty);L^2(0,1))$
\begin{align}\label{eq:solx}
x(t,\xi) = \begin{cases}
f(\xi) - t, \qquad &\text{if } f(\xi)\geq 1+t, \\
\e^{-t}f(\xi), \qquad &\text{if } f(\xi) \in (-1,1), \\
f(\xi) + t, \qquad &\text{if } f(\xi) \leq -1-t, \\
\e^{f(\xi)-1-t}, \qquad &\text{if } f(\xi)\in [1,1+t), \\
-\e^{1-t-f(\xi)}, \qquad &\text{if } f(\xi)\in(-1-t,-1],
\end{cases} 
\end{align}
which can be derived by solving \eqref{System sat0} for fixed $\xi$ as simple ODE.
We will show that there exists a sequence $(f_n)_n\in L^2(0,1)$  with $\|f_n\|_{D(A)} = \|f_n\|_{L^2(0,1)} = 1$ such that for all $t>0$ there exists an $n\in\N$ such that $\|x_n(t)\|_{L^2(0,1)}>\frac{1}{2}$ where $x_n$ denotes the corresponding solution of \eqref{System sat0} with initial function $f_n$. For this purpose we will only consider the restriction of $x_n$ to $\{\xi\in [0,1] \ | \ f(\xi)\geq 1+t\}$ and define
\begin{align*}
f_n(\xi) \coloneqq \frac{1}{\sqrt{n}}\xi^{-\alpha_n}
\end{align*}
with $\alpha_n\coloneqq \frac{1}{2}\left(1-\frac{1}{n}\right)$. Clearly, $f_n\in L^2(0,1)$, $\|f_{n}\|_{L^{2}}=1$ and $f_n$ is decreasing. 
Note that the equation $f_{n}(\xi)=1+t$ has a unique solution $\xi$ for fixed $n$ and $t$ which is given by
\begin{align*}
\xi=\xi_{t,n} := \frac{1}{(\sqrt{n}(1+t))^\frac{1}{\alpha_n}}.
\end{align*}
Therefore,
$\{\xi\in [0,1] \ | \ f_n(\xi)\geq 1+t\} = \{\xi\in [0,1] \ | \ \xi\leq \xi_{t,n}\}$.
Hence,
\begin{align*}
\|x_n(t)&\|_{L^2(0,1)}^2 \geq \int_0^{\xi_{t,n}} x_n(t,\xi)^2 \mathrm{d}\xi \\
&= \int_0^{\xi_{t,n}} \left(f_n(\xi)-t\right)^2 \mathrm{d}\xi \\
&= \int_0^{\xi_{t,n}} \left(\frac{1}{\sqrt{n}}\xi^{-\alpha_n}-t\right)^2 \mathrm{d}\xi \\
&= \frac{1}{n}\int_0^{\xi_{t,n}} \xi^{-2\alpha_n} \mathrm{d}\xi - \frac{2t}{\sqrt{n}}\int_0^{\xi_{t,n}} \xi^{-\alpha_n} \mathrm{d}\xi + \int_0^{\xi_{t,n}} t^2\mathrm{d}\xi \\
&= \frac{1}{n}\frac{1}{1-2\alpha_n}\xi_{t,n}^{1-2\alpha_n} - \frac{2t}{\sqrt{n}}\frac{1}{1-\alpha_n}\xi_{t,n}^{1-\alpha_n} + t^2\xi_{t,n} \\
&= n^{\frac{1}{1-n}}(1+t)^{\frac{2}{1-n}}-\frac{1}{n+1}2n^\frac{1}{1-n}2t(1+t)^{\frac{1+n}{1-n}}  + n^{\frac{n}{1-n}}t^2(1+t)^{\frac{2n}{1-n}}.
\end{align*}
Taking the limit $n\rightarrow \infty$ we conclude
\begin{align*}
\lim_{n\rightarrow\infty} \|x_n(t)\|_{L^2(0,1)}^2 \geq 1.
\end{align*}
Thus the solution of System  \eqref{System sat0} does not converge uniformly to $0$ with respect to the norm or graph norm of the initial value, so the system is neither semi-globally exponentially stable in $D(A)$ nor UGAS.
\hfill$\square$
\end{proof}

Note, that System \eqref{System SLD} from Proposition \ref{Thm counterexmp} is (GAS) for $d=0$ by \cite[Theorem 2]{MaAnPr17}.
After we have seen that \eqref{30} is necessary to conclude semi-global exponential stability in $D(A)$ in Proposition \ref{Thm SGES} ii), one may ask whether ``more stability'' can in fact be expected.
The following theorem shows that  
UGAS of System \eqref{System L} together with the hypotheses in Proposition \ref{Thm SGES} ii) are not sufficient to guarantee UGAS of System \eqref{System SLD}. 

\begin{theorem}\label{theo:sat}
Let $X=U=L^2(0,1)$, $B=I$, $S=L^\infty(0,1)$, $\sigma=\sat$ and
 \[A=\frac{d}{d\xi},\qquad D(A)=\{y\in H^1(0,1)\mid y(0)=y(1)\}.\]
Then the following assertions hold.
\begin{enumerate}
	\item System  \eqref{System L} is UGAS and the hypothesis of Proposition \ref{Thm SGES} ii) holds, 
	\item System \eqref{System SLD} is semi-globally exponentially stable in $D(A)$,
	\item System \eqref{System SLD} is neither UGAS nor ISS.
\end{enumerate}
\end{theorem}

We note, that System \eqref{System SLD} of Theorem \ref{theo:sat} equals \eqref{System sat}. Further, in \cite[Thm.~1]{MaChPr18} it has been wrongly stated that the saturated system is UGAS. 
\begin{proof}
It is easy to see that System  \eqref{System L} is UGAS. Since $A$ is dissipative, it follows that $P=I$ solves \eqref{prop1eq1} for $\tilde{A}=A-BB^*=A-I$. Trivially, $P$ maps $D(A)$ to $D(A)$. Condition \eqref{30} is satisfied because $H^1(0,1)$ is continuously embedded in $L^\infty(0,1)$. Hence, \eqref{System SLD} is semi-globally exponentially stable in $D(A)$ by Proposition \ref{Thm SGES} and the fact that $\sigma$ is globally Lipschitz continuous. This shows Assertions i) and ii). To see iii), note that $A$ generates the periodic shift semigroup on $L^2(0,1)$. By extending the initial function $f$ periodically to $\R_+$, the unique mild solution $y\in C([0,\infty);L^2(0,1))$ of \eqref{System SLD} is given by 
\begin{align*}
y(t,\xi) = x(t,\xi+t),
\end{align*} 
where $x$ is defined in \eqref{eq:solx}. By the particular form of \eqref{eq:solx}, this implies that 
\begin{align*}
\|x(t)\|_{L^2(0,1)} = \|y(t)\|_{L^2(0,1)}
\end{align*}
holds for all $t\geq0$. We can therefore choose the same sequence $(f_n)_n\in L^2(0,1)$ with $\|f_n\|_{L^2(0,1)}=1$ as in the proof of Proposition \ref{Thm counterexmp} in order to conclude
\begin{align*}
\lim_{n\rightarrow\infty} \|y_n(t)\|_{L^2(0,1)}^2 \geq 1.
\end{align*}
This shows that System  \eqref{System SLD} is not UGAS and thus not ISS.
\hfill$\square$
\end{proof}

An important tool for the verification of ISS of System \eqref{System SLD} are ISS Lyapunov functions.

\begin{definition}
Let $U_{r}=\{x\in X\colon \|x\|\leq r\}$ and $r\in(0,\infty]$. Let $\mathcal{U}$ be either $C(0,\infty;U)$ or $L_{{\rm loc}}^{\infty}(0,\infty;U)$. A continuous function $V:U_{r}\rightarrow\R_{\geq0}$ is called an {\em LISS Lyapunov function} for \eqref{System SLD} with respect to $\mathcal{U}$, if there exists $\psi_1,\psi_2,\alpha,\rho\in\K_\infty$, such that for all $x_0\in U_{r}$, $d\in \mathcal{U}$, $\|d\|_{L^\infty(0,\infty;U)}\le r$,
\begin{align*}
\psi_1(\|x_0\|_X) \leq V(x_0) \leq \psi_2(\|x_0\|_X)
\end{align*}
and
\begin{align}\label{24}
\dot{V}_d(x_0):=\limsup_{t\searrow0}\frac{1}{t}\big(V(x(t))-V(x_0)\big)\leq-\alpha(\|x_0\|_X)+\rho(\|d\|_{L^\infty(0,\infty;U)}),
\end{align}
where $x$ is the mild solution of \eqref{System SLD} with initial value $x_0$ and disturbance $d$. If $r=\infty$, then $V$ is called an {\em ISS Lyapunov function}.
\end{definition}

Note that our definition of an ISS Lyapunov function corresponds to the one of a ``coercive ISS Lyapunov function in dissipative form'' in the literature, \cite{MiroPrie19}.
By  \cite[Thm.~1]{DaM13}, see also \cite[Thm.~2.18]{MiroPrie19}, the existence of an (L)ISS Lyapunov implies (L)ISS for a large class of control systems which, in particular have to satisfy the ``boundedness-implies-continuation'' property (BIC). System \eqref{System SLD} with an admissible feedback function and continuous, or, more generally, piecewise continuous disturbances $d$ belongs to this class, which allows to infer (L)ISS from the existence of a Lyapunov function. To see this, note in particular that the (BIC) property is satisfied by classical results on semilinear equations, \cite[Prop.~4.3.3]{cazenaveharaux98} or \cite[Thm.~6.1.4]{pazy83}.\newline
In the following we will infer ISS by constructing Lyapunov functions. 

\begin{theorem}\label{thm:UGAS} Suppose that there exists $\alpha>0$ such that $\|T(t)\|\leq \mathrm{e}^{-\alpha t}$ for all $t>0$
and let $\sigma$ be an admissible feedback function. Then the function
\begin{align*}
V(x) = \|x\|_X^2,\quad x\in X,
\end{align*}
is an ISS Lyapunov function for \eqref{System SLD} with respect to $C(0,\infty;U)$ and System \eqref{System SLD} is ISS with respect to $C(0,\infty;U)$.
\end{theorem}
\begin{proof}
Let $x\in C(0,t_1;X)$ be the mild solution of \eqref{System SLD} with initial value $x_0\in D(A)$ and disturbance $d\in C(0,\infty;U)$. Let $y\in C(0,t_2;X)$ be the mild solution  of the system
\begin{align*}
\begin{cases}
\dot{y}(t) = Ay(t)-B\sigma\big(B^*y(t)+\tilde{d}(t)\big)\\
y(0)  = y_0
\end{cases}
\end{align*}
with $\tilde{d}\in C(0,\infty;U)$ and $y_0\in X$. Then there exists an $r>0$ such that 
\[\max\{\|B^*x(s)+d(s)\|_U,\|B^*y(s)+\tilde{d}(s)\|_U,\|B^*x(s)\|_U\mid s\in [0,\min\{t_1,t_2\}]\}<r\]
 because $x$, $y$, $d$ and $\tilde{d}$ are continuous. Thus we have for $t\in [0,\min\{t_1,t_2\})$
\begin{align*}
\|x(t)-y(t)\| &\leq \|x_0-y_0\|  + \int_0^t\|B\|k_r\big(\|B\|\|x(s)-y(s)\| + \|d(s)-\tilde{d}(s)\|\big) \,\mathrm{d}s.
\end{align*}
Applying Gronwall's inequality yields
\begin{align}\label{eq:Gronwall}
\|x(t)-y(t)\| \leq \left(\|x_0-y_0\| + \int_0^t\|B\|k_r\|d(s)-\tilde{d}(s)\|\,\mathrm{d}s\right)e^{t\|B\|^2k_r}.
\end{align}
Let us for a moment  assume that $d$ ist Lipschitz continuous with Lipschitz constant $L$. We will prove that $x$ is right-differentiable. For $0<h<t_1-t$ we can write $x(t+h)$ in the form
\begin{align*}
x(t+h) &= T(t+h)x_0-\int_0^{t+h}T(t+h-s)B\sigma\big(B^*x(s)+d(s)\big)\,\mathrm{d}s\\
&= T(t)x(h) - \int_0^tT(t-s)B\sigma\big(B^*x(s+h)+d(s+h)\big)\,\mathrm{d}s.
\end{align*}
Thus $x$ at time $t+h$ equals the mild solution $y$ of
\begin{align}\label{eq:systemy}
\begin{cases}
\dot{y}(t) = Ay(t)-B\sigma\big(B^*y(t)+d(t+h)\big)\\
y(0) = x(h)
\end{cases}
\end{align}
at time $t$. Hence, by \eqref{eq:Gronwall} we obtain
\begin{equation}\label{eq10}
\|x(t+h)-x(t)\| \leq \big(\|x(h)-x_0\|+\|B\|k_rLht\big)e^{t\|B\|^2k_r}.
\end{equation}
Note that
\begin{align*}
\frac{x(h)-x_0}{h} = \frac{T(h)x_0-x_0}{h}-\frac{1}{h}\int_0^hT(h-s)B\sigma\big(B^*x(s)+d(s)\big)\,\mathrm{d}s
\end{align*}
converges to $Ax_0-B\sigma\big(B^*x_0+d(0)\big)$ as $h\searrow0$ since $x_0\in D(A)$ and $\sigma$, $x$ and $d$ are continuous. Therefore, by \eqref{eq10}, we deduce
\begin{align}\label{eq:bdd}
\limsup_{h\searrow 0} \frac{\|x(t+h)-x(t)\|}{h} < \infty.
\end{align}
By the definition of the mild solution we have that
\begin{align*}
\frac{T(h)-I}{h}x(t) = \frac{x(t+h)-x(t)}{h}+ \frac{1}{h}\int_t^{t+h}T(t+h-s)B\sigma\big(B^*x(s)+d(s)\big)\mathrm{d}s.
\end{align*}
Again by continuity of  $\sigma$, $x$ and $d$ we have that 
\begin{align*}
\lim_{h\searrow0}\frac{1}{h}\int_t^{t+h}T(t+h-s)B\sigma\big(B^*x(s)+d(s)\big)\,\mathrm{d}s=B\sigma\big(B^*x(t)+d(t)\big).
\end{align*}
Combining this with \eqref{eq:bdd} shows that \[x(t)\in \{z\in X \:|\:\limsup_{h\searrow0}\frac{1}{h}\|T(h)x-x\|<\infty\},\] which means that 
 $x(t)$ is an element of the Favard space of the semigroup and because $X$ is reflexive, we can conclude that $x(t)\in D(A)$, \cite[Cor.~II.5.21]{engelnagel99}. 
 This implies that $x$ is right-differentiable at $t$ with
\begin{align*}
\lim_{h\searrow 0} \frac{x(t+h)-x(t)}{h} = Ax(t)-B\sigma\big(B^*x(t)+d(t)\big).
\end{align*}
As $V(x) = \|x\|^2$, we hence obtain for the Dini derivative \[D^{+}V(x(\cdot))(t)=\limsup_{h\searrow 0}\frac{1}{t}\big(V(x(t+h))-V(x(t))\big)\] that
\begin{align}
 D^{+}V(x(\cdot))(t)&= 2\Re(\langle Ax(t),x(t)\rangle_X - \langle B\sigma(B^*x(t)+d(t)),x(t)\rangle_X) \notag\\
&\leq -2\alpha\|x(t)\|^2 - \Re(\langle \sigma(B^*x(t)+d(t))-\sigma(B^*x(t)),B^{*}x(t)\rangle_X )\notag\\
&\leq -2\alpha\|x(t)\|^2 + \| \sigma(B^*x(t)+d(t))-\sigma(B^*x(t))\| \, \|B^{*}x(t)\|\notag\\
&\leq -2\alpha\|x(t)\|^2 + k_r\|d(t)\|\,\|B\|\,\|x(t)\|,\label{eq:crucial}
\end{align}
where we used that $-\Re\langle \sigma(B^{*}x),B^{*}x\rangle\leq 0$ by Property (i) and (ii) of admissible feed\-back functions and the local Lipschitz condition for $\sigma$. 
By \cite[Cor.~A.5.45]{Curtain:2717229} we obtain
\begin{align}\label{eq:CZ}
V(x(t+h))-V(x(t)) \leq \int_t^{t+h}-2\alpha\|x(s)\|^2 + k_r\|d(s)\|\,\|B\|\,\|x(s)\|\,\mathrm{d}s.
\end{align}
From \eqref{eq:Gronwall} we derive
\begin{align*}
\|x(t)-y(t)\| \leq \big(\|x_0-y_0\| + t\|B\|k_r\|d-\tilde{d}\|_{L^{\infty}(0,t;U)}\big)e^{t\|B\|^2k_r}
\end{align*}
and therefore the mild solution of \eqref{SLD} depends continuously on the initial data and the disturbance. Hence, by understanding $x(t+h)$ again as the solution of \eqref{eq:systemy} at time $t$, \eqref{eq:CZ} holds for all $x_0\in X$ and $d\in C(0,\infty;U)$ which leads to
\begin{align*}
\dot{V}_{d}(x_{0})&\leq -2\alpha\|x_{0}\|^2 + k_r\|d(0)\|\,\|B\|\,\|x_{0}\| \\
&\leq (\eps-2\alpha)\|x_{0}\|^2 + \frac{(k_r\,\|B\|\,\|d(0)\|)^2}{\eps}
\end{align*}
for all $x_{0}\in X$, $d\in C(0,\infty;U)$ and $\eps>0$. 
Choosing $\eps<2\alpha$, this shows that $V$ is an ISS-Lyapunov function for \eqref{System SLD} which implies that \eqref{System SLD} is ISS by \cite[Thm.~2.18]{MiroPrie19}.
\hfill$\square$
\end{proof}

\begin{remark}\-\label{rem:thm}
\begin{enumerate}
\item Recall that the semigroup generated by $A$ in Theorem \ref{theo:sat} was not exponentially stable. Theorem \ref{thm:UGAS}  shows that this is not accidental. 
\item \label{rem:thm2}
Note that the assumption on the semigroup made in Theorem \ref{thm:UGAS} is strictly stronger than the condition that $(T(t))_{t\geq0}$ is an exponentially stable contraction semigroup as can be seen e.g.\ for a nilpotent shift-semigroup on $X=L^{2}(0,1)$.
It is a simple consequence of the Lumer–Phillips theorem that the following assertions are equivalent for a semigroup $(T(t))_{t\geq0}$ generated by $A$ and some constant $\omega>0$.
\smallskip
\begin{enumerate}
\item $\Re\langle Ax,x\rangle\leq -\omega\|x\|^{2}$ all $x\in D(A)$.
\item $\sup_{t>0}\|e^{\omega t}T(t)\|\leq 1$.
\item $P=\frac{1}{\omega} I$ solves  $\Re\langle Ax,Px\rangle \leq-\langle x,x\rangle$, for all $x\in D(A)$.
\end{enumerate}
\smallskip
However, we also remark that the above condition is satisfied for a large class of examples, such as in the case when $A$ is a normal operator. 
\item It is natural to ask whether Theorem \ref{thm:UGAS} holds when $A$ is merely assumed to generate an exponentially stable semigroup. 
However, it is unclear how to use the structural assumptions on $\sigma$ in the general case. 
On the other hand, the assumption on the semigroup in Theorem \ref{thm:UGAS} implies that $P=I$ satisfies \eqref{prop1eq1} in Proposition \ref{Thm SGES} ii).
\item An inspection of the proof shows that Theorem \ref{thm:UGAS} can be generalized to piecewise continuous or regulated functions $d:[0,\infty)\to U$. 
\end{enumerate}
\end{remark}

Locally linear admissible feedback functions yield LISS Lyapunov functions.

\begin{theorem}\label{Thm_LISS}
Let \eqref{System L} be UGAS with $M,\omega>0$ such that $\|\wt{T}(t)\|\leq M\e^{-\omega t}$ for all $t\geq0$ and  let $\sigma$ be an admissible feedback function with $\sigma(u) = u$ for all $\|u\|_U\leq\delta$ and some $\delta>0$. Then \eqref{System SLD} is LISS with  Lipschitz continuous LISS Lyapunov function $V(x)\coloneqq \max_{s\geq0}\|\e^{\frac{\omega}{2} s}\wt{T}(s)x\|_X$.
\end{theorem}

\begin{proof}
Let $\|x_0\|_X\leq\|B\|^{-1}\delta$ and $r\coloneqq \max\{\|B^*x(s)\|_U,\|B^*x(s)+d(s)\|_U \mid s\in[0,t]\}$ for some $t>0$.
We can rewrite (\ref{System SLD}) in the form
\begin{align*}
\begin{cases}
\dot{x}(t) = \wt{A}x(t)+B\big(B^*x(t)-\sigma(B^*x(t)+d(t))\big), \\
x(0) = x_0.
\end{cases}
\end{align*}
Hence, the mild solution satisfies
\begin{align*}
x(h) = \wt{T}(h)x_0 + \int_0^h\wt{T}(h-s)B\big(B^*x(s)-\sigma(B^*x(s)+d(s)\big)\mathrm{d}s.
\end{align*}
Denoting the integral by $I_{h}$, we have
\begin{align*}
\limsup_{h\searrow0}\frac{1}{h}\|I_h\|_X &\leq \limsup_{h\searrow0}\frac{1}{h}\bigg(\int_0^hM\|B\|\|B^*x(s)-\sigma(B^*x(s))\|_U\mathrm{d}s\\
& \qquad \qquad \  + \int_0^hM\|B\|\|\sigma(B^*x(s))-\sigma(B^*x(s)+d(s))\|_U\mathrm{d}s\bigg)\\
&\leq \ M\|B\|\|B^*x_0-\sigma(B^*x_0)\|_U + M\|B\|k_r\|d\|_{L^\infty(0,\varepsilon;U)} \\
&= \ M\|B\|k_r\|d\|_{L^\infty(0,\varepsilon;U)},
\end{align*}
where the continuity of $x$, the Lipschitz continuity of $\sigma$ as well as the condition $\sigma(u)=u$ if $\|u\|\le \delta$ have been used. \\
With $\|x\|\leq V(x)\leq M\|x\|$ and $V\big(\wt{T}(t)x\big) \leq \e^{-\frac{\omega}{2} t}V(x)$ for all $x\in X$ we obtain
\begin{align*}
\dot{V}_d(x_0) &= \limsup_{h\searrow0}\frac{1}{h}\big(V(\wt{T}(h)x_0+I_h)-V(x_0)\big) \\
&\leq \limsup_{h\searrow0}\frac{1}{h}\left(\e^{-\frac{\omega}{2} h}-1\right)V(x_0) + M\limsup_{h\searrow0}\frac{1}{h}\|I_h\|_X \\
&\leq -\frac{\omega}{2} \|x_0\|_X + M^2\|B\|k_r\|d\|_{L^\infty(0,\varepsilon;U)}
\end{align*} 
for every $\varepsilon>0$.
The Lipschitz continuity of $V$ follows from
\begin{align*}
|V(x)-V(y)| &\leq |\max_{s\geq0}\|\e^{\frac{\omega}{2} s}\wt{T}(s)x\|- \max_{s\geq0}\|\e^{\frac{\omega}{2} s}\wt{T}(s)y\|| \\
&\leq \max_{s\geq0}\|\e^{\frac{\omega}{2} s}\wt{T}(s)(x-y)\| \\
&\leq M\|x-y\|,
\end{align*}
for all $x,y\in X$.
Applying \cite[Theorem 4]{Mi16} yields local input-to-state stability of \eqref{System SLD}. \hfill$\square$
\end{proof}

Note that Property iii) of Definition \ref{def:admfeedback} has not been used in the proof of Theorem \ref{Thm_LISS}. 

\section{Conclusion}
In this note we have continued the study of ISS for saturated feedback connections of linear systems. Theorem \ref{theo:sat} states that ISS cannot be concluded from uniform exponential stability of the unsaturated closed-loop and stability of the (undisturbed) open-loop linear system
\[\dot{x}(t)=Ax(t)\]
(i.e.\ the semigroup generated by $A$ is bounded). However, the conclusion does hold under more assumptions on $A$; namely that 
$\Re\langle Ax,x\rangle \leq -\alpha\|x\|^{2}$ for some $\alpha>0$ and all $x\in D(A)$, see Theorem \ref{thm:UGAS}. The latter property can be seen as some kind of quasi-contractivtiy of the semigroup combined with exponential stability. This condition seems to be crucial for the proof, see Remark \ref{rem:thm}. The question remains whether the result could be generalized to more general semigroups, e.g.\ such as contractive semigroups which are exponentially stable, but do not satisfy the above mentioned quasi-contractivity. Note, however, that the assumption that $A$ generates a contraction semigroup  seems to be to essential to employ dissipativity of the nonlinear system.\\
Another task for future research is the step towards unbounded operators $B$, prominently appearing in boundary control systems. As our techniques and also the ones used in existing results for ISS on saturated systems, seem to heavily rely on the boundedness of $B$, this may require a different approach or more structural assumptions on $A$.

\section*{Ackowledgements}
The authors thank Hans Zwart for fruitful discussions on the proof of Proposition \ref{Thm counterexmp} during a visit of the third author at the University of Twente.

\end{document}